\documentclass[reqno]{article}
\usepackage{amsmath,amsfonts,amssymb,amsthm}

\newtheorem{theorem}{Theorem}
\newtheorem{lemma}[theorem]{Lemma}

\theoremstyle{definition}

\renewcommand\le{\leqslant}
\renewcommand\ge{\geqslant}

\newcommand\eps{\varepsilon}
\newcommand\la{\lambda}
\newcommand\las{\lambda_*}

\renewcommand\Pr{{\mathbb P}}
\newcommand\E{{\mathbb E}}
\newcommand\Var{{\mathrm{Var}}}
\newcommand\dto{\overset{\mathrm{d}}{\to}}

\newcommand\tX{{\widetilde X}}

\newcommand\op{o_{\mathrm{p}}}
\newcommand\Op{O_{\mathrm{p}}}
\newcommand\cF{\mathcal{F}}

\newcommand\bb[1]{\bigl(#1\bigr)}
\newcommand\Bb[1]{\Bigl(#1\Bigr)}
\newcommand\bp{{\mathfrak X}}

\begin{document}
\title{Asymptotic normality of the size of the giant component in a random hypergraph}
\author{B\'ela Bollob\'as%
\thanks{Department of Pure Mathematics and Mathematical Statistics,
Wilberforce Road, Cambridge CB3 0WB, UK and
Department of Mathematical Sciences, University of Memphis, Memphis TN 38152, USA.
E-mail: {\tt b.bollobas@dpmms.cam.ac.uk}.}
\thanks{Research supported in part by NSF grants DMS-0906634,
 CNS-0721983 and CCF-0728928, ARO grant W911NF-06-1-0076,
and TAMOP-4.2.2/08/1/2008-0008 program of the Hungarian Development Agency}
\and Oliver Riordan%
\thanks{Mathematical Institute, University of Oxford, 24--29 St Giles', Oxford OX1 3LB, UK.
E-mail: {\tt riordan@maths.ox.ac.uk}.}}
\date{December 15, 2011}
% (compiled \today)}
\maketitle

\begin{abstract}
Recently, we adapted random walk arguments based on work
of Nachmias and Peres, Martin-L\"of, Karp and Aldous to give a simple proof
of the asymptotic normality of the size
of the giant component in the random graph $G(n,p)$ above the phase
transition. Here we show that the same method applies to 
the analogous model of random $k$-uniform hypergraphs, establishing
asymptotic normality throughout the (sparse) supercritical regime.
Previously, asymptotic normality was known only towards the
two ends of this regime.
\end{abstract}

\section{Introduction and results}

Let $H_k(n,p)$ denote the random $k$-uniform hypergraph with vertex
set $[n]=\{1,2,\ldots,n\}$ in which each of the $\binom{n}{k}$
possible edges is present independently with probability $p$. Thus
$H_2(n,p)$ is the classical random graph $G(n,p)$. Our aim here is to
study the component structure of $H_k(n,p)$, in particular the distribution
of the size of the largest component above the `phase transition'.

Before turning to the details, let us note that the notion of a
`component' in a $k$-uniform hypergraph $H_k$ can be interpreted in a
number of ways. For any $1\le r\le k-1$, one could consider two edges
to be `connected' if they share at least $r$ vertices (or perhaps
exactly $r$ vertices), and use this notion to define the components of
$H_k$. Moreover, the size of a component could then be measured in a
number of ways -- either by the number of vertices that it contains,
or (probably more naturally) by the number of $r$-sets of vertices, or
the number of edges. In the rest of the paper we consider $r=1$. It seems
that other values of $r$ have received very little
attention, although the corresponding notions of `cycle' (with $r=1$
being `loose' and $r=k-1$ being `tight') have been studied
extensively.  Note that for hypergraphs derived from $k$-cliques in
random graphs, questions about components defined using $r=k-1$ were
raised by Der{\'e}nyi, Palla and Vicsek~\cite{DPV}; such components
were studied for all $1\le r\le k-1$ in~\cite{cliqueperc}.  We shall
say nothing further about the case $r\ge 2$, except to note that
(greatly simplified versions of) the branching process arguments
in~\cite{cliqueperc} presumably show that the threshold for the
emergence of a giant component is at $p\sim n^{r-k}(k-r)!/(\binom{k}{r}-1).$

For the rest of the paper we take $r=1$, i.e., we say that two vertices
are \emph{connected} in a $k$-uniform hypergraph $H$ if they are connected in the graph
obtained by replacing each edge by a copy of $K_k$, and take the \emph{components}
of $H$ to be the maximal sub-hypergraphs in which all vertices are connected.
For reasons that will become clear below we write $p=p(n)$ as $\lambda(k-2)!n^{-k+1}$,
where $\la=\la(n)$; when $k=2$ this reduces to $p=\la/n$. Our main aim is to study the
number $L_1$ of vertices in the largest component of $H_k(n,p)$ in the supercritical regime,
i.e., when $(\la-1)n^{1/3}\to\infty$. We also prove a result for the critical
regime, where $(\la-1)n^{1/3}$ is bounded.

For $k=2$, very detailed results of this type are known.
Pittel and Wormald~\cite{PWio} and Luczak and \L uczak~\cite{LuczakLuczak} showed,
%Theorems 14 and 17 in the latter
%Local limit modulo caveat for several parameters in former
in each case as part of a much stronger and/or more general result, that throughout
the supercritical regime, $L_1$ is asymptotically normal: centralized and scaled
appropriately, it converges in distribution to a standard normal distribution.
The special case where $\la>1$ is constant was proved earlier by Stepanov~\cite{Stepanov_conn}.
For hypergraphs, where $k\ge3$ is fixed, much less is known: Karo\'nski and {\L}uczak~\cite{KL_giant} proved
strong results (a local limit theorem) in the barely supercritical phase,
when $(\la-1)^3n$ tends to infinity but more slowly than $\log n/\log\log n$.
At the other end of the range, Behrisch, Coja-Oghlan and Kang~\cite{BC-OK1} proved
a local limit theorem when $\la>1$ is fixed. Here we shall prove asymptotic
normality throughout the supercritical regime, for all $k\ge3$ fixed.
Note that our main result, while less precise than those of~\cite{KL_giant,BC-OK1},
has a much greater range of applicability. The proof is a (to us surprisingly) simple
adaptation of the argument we gave for the case $k=2$ in~\cite{BR_walk}, itself
based on exploration and martingale arguments using ideas of Nachmias and Peres~\cite{NP_giant},
Martin-L\"of~\cite{ML86}, Karp~\cite{Karp} and Aldous~\cite{Aldous}.

Given $\la>1$, let $\las$ be the `dual branching process parameter', defined by $\las<1$ and
\[
 \las e^{-\las} = \la e^{-\la}.
\]
Writing $\bp_\mu$ for the Galton--Watson branching process in which the offspring
distribution is Poisson with mean $\mu$, it is well known that conditioning
$\bp_{\la}$ on extinction gives $\bp_{\las}$.
Let $\rho_{\la}=\rho_{2,\la}$ denote the survival probability of $\bp_\la$,
so $\rho_{\la}>0$ may be defined by
\begin{equation}\label{rldef}
 1-\rho_{\la} = e^{-\la \rho_\la},
\end{equation}
and satisfies $\las=\la\rho_\la$.
Finally, for $k\ge3$ define $\rho_{k,\la}$ by
\begin{equation}\label{rkldef}
 1-\rho_{k,\la} = (1-\rho_{\la})^{1/(k-1)};
\end{equation}
it is easy to see that $\rho_{k,\la}$ is the survival probability of a certain 
branching process naturally associated to $H_k(n,p)$, where $p=\la (k-2)! n^{-k+1}$.

As usual, we say that a sequence $(E_n)$ of events holds \emph{with
high probability} or \emph{whp} if $\Pr(E_n)\to 0$ as $n\to\infty$.
If $(X_n)$ is a sequence of random variables and $f(n)$ is a deterministic
function, then $X_n=\op(f(n))$ means
that $X_n/f(n)$ converges to $0$ in probability, i.e., that for any constant
$\eps>0$, $|X_n|\le\eps f(n)$ holds whp. Later, we shall also use $X_n=\Op(f(n))$
to mean that $X_n/f(n)$ is bounded in probability, i.e., for any $\eps>0$
there is a $C$ such that for all (large enough) $n$ we have $\Pr(|X_n|\ge C f(n))\le \eps$.

Writing $L_1(H)$ for the maximum number of vertices in any component of a hypergraph $H$,
Coja-Oghlan, Moore and Sanwalani~\cite{C-OMS} showed that
if $k\ge3$ and $\la>1$ are fixed and $p=p(n)=\la (k-2)! n^{-k+1}$, then
$L_1(H_k(n,p))=\rho_{k,\la}n+\op(n)$. 
(This result was certainly known as `folklore'
before this.) Our main result concerns the limiting distribution of the $\op(n)$
term, and applies throughout the supercritical regime.

Given $k\ge 2$ and $\la>1$, let
\begin{equation}\label{skldef}
 \sigma_{k,\la}^2   = \frac{ \la(1-\rho)^2 - \las(1-\rho) +\rho(1-\rho) }{(1-\las)^2} n,
\end{equation}
where $\rho=\rho_{k,\la}$. It is well known that when $\la=1+\eps$ and $\eps\to 0$,
then $\rho_{\la}\sim 2\eps$. From \eqref{rkldef} it follows that
\[
 \rho_{k,\la} \sim \frac{2\eps}{k-1},
\]
and thus $\sigma_{k,\la}^2\sim 2\eps^{-1}n$. Expanding $\las$ and thus $\rho_{\la}$
and hence $\rho_{k,\la}$ further as series in $\eps=1-\la$, it is easy to check
that in fact
\[
 \sigma_{k,\la}^2 = \left(2\eps^{-1} +  \frac{2(k-4)}{k-1} +O(\eps)\right)n.
\]
Thus, although the leading term does not depend on $k$, the next term does.

\begin{theorem}\label{th+}
Let $k\ge 3$ be fixed, and let
\[
 p=p(n)= \la  (k-2)!  n^{-k+1},
\]
where $\la=\la(n)$ is bounded and $(\la-1)^3n\to\infty$.
Then
\[
 \frac{L_1(H_k(n,p))-\rho_{k,\la}n}{\sigma_{k,\la}} \dto N(0,1),
\]
where $\dto$ denotes convergence in distribution, $N(0,1)$ is a standard
normal random variable, and $\rho_{k,\la}$ and $\sigma_{k,\la}$ are defined
in \eqref{rkldef} and \eqref{skldef}.
\end{theorem}

As a by-product of our proof, we obtain an analogue of the result
of Aldous~\cite{Aldous} giving the limiting distribution of the rescaled
large component sizes inside the scaling window of the phase transition.
Define a stochastic process $W^\alpha(s)$ (a random function on $[0,\infty)$) by
\[
 W^\alpha(s)= W(s) +\alpha s-s^2/2,
\]
where $W(s)$ is a standard Brownian motion. As in~\cite{Aldous},
define an \emph{excursion} of this process to be a maximal
interval on which $W^\alpha$ exceeds its previous minimum value,
and let $(|\gamma_i|)_{i\ge 1}$ denote the lengths of the excursions sorted
into decreasing order. (Aldous shows that this makes sense with probability 1.)

\begin{theorem}\label{th_crit}
Let $k\ge 3$ be fixed, and let $p=p(n)= \la  (k-2)!  n^{-k+1}$
where $\la=\la(n)$ satisfies
\[
 (\la-1)^3n\to(k-1)^2\alpha^3
\]
for some $\alpha\in\mathbb{R}$. Then, for any fixed $r$, writing $L_r$
for the number of vertices in the $r$th largest component
of $H(n,p)$, the sequence $((k-1)^{1/3}n^{-2/3}L_i)_{i=1}^r$
converges in distribution to $(|\gamma_i|)_{i=1}^r$
where $|\gamma_i|$ is defined as above.
\end{theorem}

The slightly strange scaling is chosen to match the graph case: the conclusion
is that under these assumptions, up to a $(k-1)^{1/3}$ scaling factor,
the large component sizes have the same limiting distribution
as in the random graph $G(n,(1+\alpha n^{-1/3})/n)$.

\section{Proofs}

In this section we shall prove Theorems~\ref{th+} and~\ref{th_crit}.
The arguments, which closely follow those in~\cite{BR_walk}, require a little preparation.

Let $H=H_k(n,p)$ be the random $k$-uniform hypergraph defined in the introduction.
Our proofs are based on an algorithm for `exploring' the components of $H$
in $n$ steps. For $0\le t\le n$, `time
$t$' refers to the situation after $t$ steps, so
step $t$ goes from time $t-1$ to time $t$. In step $t$ we shall
`explore' a vertex $v_t$, meaning that we reveal all edges
incident with $v_t$ but not with any previously explored
vertices. Noting that one vertex is explored in each step, 
this means that, however $v_t$
is chosen, each of the $\binom{n-t}{k-1}$ possible
edges containing $v_t$ and not containing any previously explored
vertices will be present with probability $p$, independently of the others
and of the history.

More precisely, as in~\cite{BR_walk}, at time $t$ every
vertex is either `explored', `active' or `unseen'. We write $A_t$ and $U_t$ 
for the numbers of active and unseen vertices; exactly $t$ vertices
will be explored by time $t$, so $A_t+U_t=n-t$.
At time $t=0$, we have $A_0=0$ and $U_0=n$.
Fix an order on the vertices. In step $1\le t\le n$ we choose
$v_t$ to be the first active vertex (at time $t-1$), if there are any;
otherwise $v_t$ is the first unseen vertex.
In the latter case
we say that we `start a new component' in step $t$.
In step $t$ we reveal all edges containing $v_t$ and not containing any explored vertex.
Let $\eta_t$ be the number of unseen vertices other than $v_t$
in such edges. These $\eta_t$ vertices are now labelled active (at time $t$),
and $v_t$ is labelled as explored.
It is easy to check that the process reveals the components of $H$
one-by-one, starting a new component in step $t$ whenever $A_{t-1}=0$.
Thus, if $0=t_0<t_1<t_2<\cdots<t_k=n$ enumerates $\{t:A_t=0\}$, then
the sequence $(t_i-t_{i-1})_{i=1}^k$
lists exactly the numbers of vertices in the components of $H$, in some order.
In particular,
\begin{equation}\label{L1}
 L_1 = \max\{t_i-t_{i-1} : 1\le i\le k\}.
\end{equation}

We shall study the random walk $(X_t)$ defined by $X_t=A_t-C_t$,
where $C_t$ is the number of new components started within the first $t$
steps. As in~\cite{BR_walk}, we have $t_i=\inf\{t:X_t=-i\}$.
In step $t$, exactly $\eta_t$ vertices change state from unseen
to active. Furthermore, one vertex $v_t$ changes its state to explored: if a new
component is started in step $t$, then $v_t$ was previously unseen, otherwise it was
active. Hence, if we start a new component, then $A_t=A_{t-1}+\eta_t$,
otherwise $A_t=A_{t-1}+\eta_t-1$. Since $A_0=C_0=0$, it follows that
\[
 X_t = A_t-C_t = \sum_{i=1}^t (\eta_i-1).
\]

Let $(\Omega,\cF,\Pr)$ be a probability space supporting our random hypergraph $H_k(n,p)$. (For
example, take $\Omega$ to be the set of all $2^{\binom{n}{k}}$ $k$-uniform hypergraphs on $[n]$,
$\cF$ to be the power-set of $\Omega$, and $\Pr$ to be the appropriate probability measure.)
Let $\cF_t\subseteq\cF$ denote the sub-sigma-field generated by all information revealed by time~$t$.
Following the strategy of~\cite{BR_walk}, the key task is to understand the distribution of $\eta_{t+1}$
given $\cF_t$. Crucially, it is only the expectation that we need to bound precisely; 
our bound on the variance can be much cruder.

Given $\cF_t$,
we know which vertex $v_{t+1}$ we are about to explore in step $t+1$.
At time $t$ there are $U_t'=U_t-1_{\{A_t=0\}}=U_t-(C_{t+1}-C_t)$ unseen vertices $u$ other than $v_{t+1}$. 
(The vertex $v_{t+1}$ is active if $A_t>0$; otherwise it is unseen.)
For each such unseen vertex $u$ there are exactly
\begin{equation}\label{ct}
 c_{t+1}=\binom{A_t+U_t-2}{k-2}=\binom{n-t-2}{k-2}
\end{equation}
potential edges containing $v_{t+1}$ and $u$ but not containing any of
the $t$ vertices previously explored. Since each such edge is present with probability $p$,
the probability that $u$ becomes active during step $t+1$ is
\[
 \pi_1=1-(1-p)^{c_{t+1}} = pc_{t+1}+O(p^2c_{t+1}^2) = \la (n-t)^{k-2}n^{-k+1}+O(1/n^2).
\]
Of course, these events are not independent for different $u$,
but this does not matter for the expectation. In particular,
\begin{equation}\label{Eeta}
 \E(\eta_{t+1}\mid\cF_t) = U_t' \pi_1 
= U_t' p c_{t+1} +O(1/n).
\end{equation}
Fortunately for the subsequent analysis, this expression depends on $U_t'$
in a linear way.

We next estimate the conditional variance of $\eta_{t+1}$ given $\cF_t$; here we 
do not need to be so accurate. Let $u_1$ and $u_2$ be distinct unseen vertices
(other than $v_{t+1}$ if that happens to be unseen).
The probability that $u_1$ and $u_2$
both become active is $\pi_2+\pi_3$, where $\pi_2$
is the probability that we find an edge containing $v_{t+1}$, $u_1$ and $u_2$,
and $\pi_3$ is the probability that this does not happen, but we find
disjoint edges activating $u_1$ and $u_2$.
Supposing that $n-t\to\infty$, then
\[
 \pi_2 = 1-(1-p)^{\binom{n-t-3}{k-3}} \sim p\binom{n-t-3}{k-3} \sim \la (k-2) (n-t)^{k-3} n^{-k+1}.
\]
Also, it is easy to check that $\pi_3\sim \pi_1^2$.  Hence,
\begin{eqnarray}
 \Var(\eta_{t+1}\mid \cF_t) &=& U_t'(U_t'-1)(\pi_2+\pi_3) + U_t'\pi_1 - (U_t'\pi_1)^2\nonumber \\
 &\sim& (U_t')^2 \pi_2 +U_t'\pi_1 \nonumber \\
  &\sim& \la (k-2) (1-t/n)^{k-3} \frac{U_t^2}{n^2} + \la (1-t/n)^{k-2}\frac{U_t}{n}. \label{var}
\end{eqnarray}
In particular, the maximum possible value satisfies
\begin{equation}\label{varbdd}
 \max_t \sup_\Omega \Var(\eta_{t+1}\mid \cF_t) \le \la(k-1) =O(1).
\end{equation}

Let $D_t=\E(\eta_t-1\mid \cF_{t-1})$. Recalling that
$U_t=n-t-A_t=n-t-(X_t+C_t)$, and noting that $U_t'=U_t-(C_{t+1}-C_t)=n-t-X_t-C_{t+1}$, from \eqref{Eeta} we 
see that
\begin{equation}\label{dt}
 D_{t+1} = pU_t'c_{t+1}-1 +O(1/n) = \alpha_{t+1} (n-t-X_t-C_{t+1})-1 +O(1/n),
\end{equation}
where 
\[
 \alpha_t = pc_t = p\binom{n-t-1}{k-2}.
\]
Set $\Delta_{t+1}=X_{t+1}-X_t-D_{t+1}$, so $\Delta_{t+1}$ is $\cF_{t+1}$-measurable
and $\E(\Delta_{t+1}\mid\cF_t)=0$, by the definition
of $D_{t+1}$. Note that
\begin{eqnarray}
 X_{t+1} &=& X_t + D_{t+1} +\Delta_{t+1} \nonumber \\
 &=& (1-\alpha_{t+1}) X_t + \alpha_{t+1}(n-t) -1 + \Delta_{t+1} - \alpha_{t+1}C_{t+1} +O(1/n).\label{Xrec}
\end{eqnarray}

We shall approximate $(X_t)$ by the sum of a deterministic sequence and a martingale.
To this end, define a deterministic sequence $(x_t)$ by $x_0=0$ and
\begin{equation}\label{xrec}
 x_{t+1} = (1-\alpha_{t+1})x_t +\alpha_{t+1}(n-t) -1.
\end{equation}
Subtracting \eqref{xrec} from \eqref{Xrec} we see that
\begin{equation}\label{Xxrec}
 X_{t+1}-x_{t+1} = (1-\alpha_{t+1})(X_t-x_t) + \Delta_{t+1}-\alpha_{t+1}C_{t+1}+E_{t+1},
\end{equation}
where $E_{t+1}$ is an `error term' with $E_{t+1}=O(1/n)$.
Defining
\[
 \beta_t=\prod_{i=1}^t (1-\alpha_i),
\]
the recurrence relation \eqref{Xxrec} may be easily solved to give
\begin{equation}\label{Xx}
 X_t-x_t = \sum_{i=1}^t \frac{\beta_t}{\beta_i} (\Delta_i-\alpha_iC_i+E_i).
\end{equation}
Note that $0<\alpha_i<1$ for each $i$, so the sequence $\beta_t$ is decreasing.

Motivated by this formula, let
\begin{equation}\label{St}
 S_t = \sum_{i=1}^t \beta_i^{-1}\Delta_i,
\end{equation}
so $(S_t)$ is a martingale with respect to $(\cF_t)$, and set
\[
 \tX_t = x_t+\beta_t S_t:
\]
this is our desired approximation for $(X_t)$.

\begin{lemma}\label{Wc}
For any $p=p(n)=O(n^{-k+1})$ we have
\[
 | X_t-\tX_t | =O(tC_t/n),
\]
uniformly in $1\le t\le n$.
\end{lemma}
\begin{proof}
From \eqref{Xx} and the definition of $\tX_t$ we have
\[ 
 X_t-\tX_t = \sum_{i=1}^t \frac{\beta_t}{\beta_i} (E_i-\alpha_i C_i).
\]
The result follows from the fact that $\beta_t/\beta_i=\prod_{j=i+1}^t(1-\alpha_j)$
is between $0$ and $1$, the bounds $E_i,\alpha_i=O(1/n)$,
and the fact that $C_i\le C_t$ for $i\le t$.
\end{proof}

We next analyze the deterministic trajectory $(x_t)$.
Setting $x_t=n-t-y_t$, the relation \eqref{xrec} can be rearranged
to give $y_{t+1}=(1-\alpha_{t+1})y_t$. Since $y_0=n-x_0=n$, we see that
$y_t=n\beta_t$, so
\begin{equation}\label{xbt}
 x_t = n-t-n\beta_t.
\end{equation}

Recall that the $\alpha_i$ are $O(1/n)$. Since there are at most
$n$ terms in the sum, it follows that
\[
 \log\beta_t = \sum_{i=1}^t \log(1-\alpha_i) = -\sum_{i=1}^t \alpha_i +O(1/n).
\]
From the definition of $\alpha_i=pc_i$ we have
\[
 \sum_{i=1}^t\alpha_i = p\sum_{i=1}^t\binom{n-i-1}{k-2} = p\left(\binom{n-1}{k-1}-\binom{n-t-1}{k-1}\right).
\]
It follows that
\begin{eqnarray}
 \log \beta_t &=& -\frac{pn^{k-1}}{(k-1)!} \left(1-(1-t/n)^{k-1}\right)+O(1/n) \nonumber\\
 &=& -\frac{\la}{k-1} \left(1-(1-t/n)^{k-1}\right)+O(1/n).\label{bsim}
\end{eqnarray}
Define the function $g=g_{k,\la}$ by
\begin{equation}\label{gdef}
  g_{k,\la}(\tau) = 1-\tau - \exp\left( -\frac{\la}{k-1} \Bb{1-(1-\tau)^{k-1}} \right)
\end{equation}
and (for compatibility with the notation in \cite{BR_walk}) set
\[
 f(t)=f_{n,k,\la}(t) = ng(t/n).
\]
Then \eqref{xbt} and \eqref{bsim} imply that
that $x_t=f(t)+O(1)$, uniformly in $0\le t\le n$. In other words,
the function $f$ or $g$ represents a (rescaled in the case of $g$) idealized
form of the deterministic approximation $(x_t)$ to $(X_t)$.

From \eqref{rldef} and \eqref{rkldef} it is easy to check that $\rho=\rho_{k,\la}$
satisfies $g(\rho)=0$.
Note that
\[
 g'(\tau) = -1 +\la(1-\tau)^{k-2} \exp\left( -\frac{\la}{k-1} \Bb{1-(1-\tau)^{k-1}} \right),
\]
so $g'(0)=\la-1$. Also, recalling that $(1-\rho)^{k-1}=1-\rho_{\la}=\las/\la$,
\begin{equation}\label{g'rho}
 g'(\rho) = -1 + \la(1-\rho)^{k-1} = -(1-\las).
\end{equation}
Furthermore,
\begin{multline*}
 g''(\tau) = \left( -\la(k-2)(1-\tau)^{k-3} - \bb{\la(1-\tau)^{k-2}}^2 \right) \\
 \exp\left( -\frac{\la}{k-1} \Bb{1-(1-\tau)^{k-1}} \right),
\end{multline*}
so $g''\le 0$. Hence $g$ is concave, so $f$ is concave.
Also, $\sup\{|g''(\tau)|:0\le\tau\le 1\}=O(1)$, so $f''(t)=O(1/n)$, uniformly in $0\le t\le n$.

Note that
\[
 g''(0)=-\bb{\la(k-2)+\la^2} = -(k-1) +O(\eps),
\]
where $\eps=\la-1$.
Since (as is easily checked)
$g'''$ is uniformly bounded, it follows that for $\tau=O(\eps)$
we have
\begin{equation}\label{gexp}
 g(\tau) = g(0) + \tau g'(0) + \tau^2 g''(0)/2 + O(\tau^3)
 = \eps \tau - (k-1) \tau^2/2 +O(\eps^3).
\end{equation}
%which easily gives $\rho\sim 2\eps/(k-1)$ and hence (from \eqref{g'rho}) $g'(\rho)\sim -\eps$.
With this preparation behind us, the proof of Theorem~\ref{th+} will follow
that in~\cite{BR_walk} very closely, so we give only an outline. We use the same
notation as in~\cite{BR_walk} whenever possible.

\begin{proof}[Proof of Theorem~\ref{th+}]
Firstly, we have $\log\beta_t=O(1)$ uniformly in $0\le t\le n$. This and \eqref{varbdd}
imply that the increments $\beta_i^{-1}\Delta_i$
of the martingale $(S_t)$ have variance $O(1)$, so $\Var(S_t)=O(t)$
for any deterministic $t=t(n)$. Hence, by
Doob's maximal inequality, $\max_{i\le t} |S_t|=\Op(\sqrt{t})$.
In particular, 
\begin{equation}\label{tX1}
 \tX_t = x_t+\beta_t S_t = f(t)+\beta_tS_t+O(1) = f(t)+\Op(\sqrt{t}).
\end{equation}

Let $\sigma_0=\sqrt{\eps n}$, let $\omega=\omega(n)$ tend to infinity
slowly (with $\omega^6=o(\eps^3 n)$), and let
$t_0=\omega\sigma_0/\eps$. Write $Z=-\inf\{X_t:t\le t_0\}$ for the number of components
completely explored by time $t_0$, and $T_0$ for the time at which
we finish exploring the last such component. Since $f'(0)=g'(0)=\eps$,
and we have the bound $X_t-\tX_t=O(tC_t/n)$ from Lemma~\ref{Wc},
the proof of \cite[Lemma 6]{BR_walk} goes through \emph{mutatis mutandis}
to show that $Z\le \sigma_0/\omega$ and $T_0\le\sigma_0/(\eps \omega)$.
Note for later that the latter bound gives $T_0=\op(\sqrt{n/\eps})$.

Writing $t_1=\rho_{k,\la} n$ (ignoring the irrelevant rounding to integers),
let $T_1$ be the first time at which we finish exploring a component after time $t_0$.
Arguing exactly as in~\cite{BR_walk} we see that $t_1-t_0\le T_1\le t_1+t_0$ holds whp,
and indeed that
\begin{equation}\label{T1}
 T_1 = t_1 + \tX_{t_1}/(1-\las) + \op(\sigma_0/\eps).
\end{equation}
(Relation \eqref{T1} takes the place of equation (21) in \cite{BR_walk}.)

We claim that for each $t\le t_1$ we have $U_t=u_t+\op(n)$, where now
\begin{equation}\label{ubound}
 u_t = n \exp\left( -\frac{\la}{k-1} \Bb{1-(1-t/n)^{k-1}} \right).
\end{equation}
Noting that $u_t=n-t-f(t)$, and recalling that $U_t=n-t-A_t=n-t-(X_t+C_t)$,
this can be deduced from the crude bound \eqref{tX1} on $\tX_t$ and Lemma~\ref{Wc}
by arguing as in~\cite{BR_walk}. Note that, as pointed out to us
(in the graph case) by Lutz Warnke, the approximate form
of this formula is easy to guess: ignoring vertices selected to start
new components, a given vertex $u$ is unseen at time $t$ if and only if
none of the potential edges containing $u$ tested during the first $t$ steps
was found to be present. There are $\binom{n-1}{k-1}-\binom{n-t-1}{k-1}$ such
edges: those containing $u$ and at least one of the $t$ explored vertices.

From \eqref{var} and \eqref{ubound}, the sum of the conditional variances
of the first $t_1$ increments of $(S_t)$ is concentrated around
\begin{equation}\label{varsum}
 \sum_{i=1}^{t_1}\beta_i^{-2}\left(\la (k-2) (1-i/n)^{k-3} (u_i/n)^2
 + \la (1-i/n)^{k-2}u_i/n\right).
\end{equation}
Although the distribution of the increments is not quite as nice as in the graph
case, it is easy to see that the conditional distribution of $\eta_t$ given $\cF_{t-1}$
is dominated by $k-1$ times a binomial random variable with mean $O(1)$. (The binomial 
random variable is the number
of edges found; each contributes a number of new unseen vertices between $0$
and $k-1$.) Thus any fixed moment of $\eta_t$ is bounded by a constant,
and this transfers to $D_t=\eta_t-\E(\eta_t\mid \cF_{t-1})$ and hence to $\beta^{-t}D_t$.
This condition (for the fourth moment) and concentration of the sum
of the conditional variances is more than enough for a martingale central limit theorem 
such as Brown~\cite[Theorem 2]{Brown}, and it follows that $S_{t_1}$,
and hence $\tX_{t_1}$ and thus $T_1$, is asymptotically normally distributed.

For the variance, 
the sum in \eqref{varsum} is well approximated by an integral, and after a slightly unpleasant
calculation one sees that
\[
 \Var(\tX_{t_1}) = \Var(\beta_tS_t) \sim \bigl( \la(1-\rho)^2 - \las(1-\rho) +\rho(1-\rho) \bigr)n = (1-\las)^2\sigma^2,
\]
where $\rho=\rho_{k,\la}$ and $\sigma=\sigma_{k,\la}$ are defined in \eqref{rkldef} and \eqref{skldef}.
Recalling \eqref{T1}, and noting that $T_0=\op(\sqrt{n/\eps})=\op(\sigma)$,
it follows that $T_1$ and thus
$T_1-T_0$ is asymptotically normal with mean $\rho n$ and variance $\sigma^2$,
so there is a component whose size has the required distribution.

Finally, it is not hard to check that there is whp no larger
component, using the fact that what remains to be explored after time $T_1$ is a
subcritical random hypergraph.  Specifically, one can either apply a
martingale argument as in Nachmias and Peres~\cite{NP_giant}, or simply
apply the subcritical case of the results proved by Karo\'nski and \L uczak~\cite{KL_giant}.
\end{proof}

\begin{proof}[Proof of Theorem~\ref{th_crit}]
Suppose that $p=\la (k-2)!n^{-k+1}$ with $\la=1+\eps$,
where $\eps=\eps(n)$ satisfies
$\eps n^{1/3}\to (k-1)^{2/3}\alpha$ as $n\to\infty$, for some $\alpha\in\mathbb{R}$ constant.
As before we consider the random walk $(X_t)$, but now only for
$t\le An^{2/3}$ for a large constant $A$. Aldous~\cite{Aldous}
shows in the graph case that, appropriately rescaled,
the process $(X_t)$ converges to a deterministic quadratic plus
a standard Brownian motion. We show the same here,
with slightly different rescaling.

The argument giving $x_t=ng(t/n)+O(1)$ above assumed only that $\la=O(1)$, 
which applies here.
Writing $t=s (k-1)^{-1/3}n^{2/3}$, using \eqref{gexp}
we see that for $t\le An^{2/3}$ we have
\begin{multline*}
 x_t/((k-1)n)^{1/3} = (k-1)^{-1/3}n^{2/3}g\bigl(s (k-1)^{-1/3}n^{-1/3}\bigr)+o(1)  \\
 = n^{1/3}\eps(k-1)^{-2/3}s - s^2/2 +o(1) = \alpha s-s^2/2+o(1).
\end{multline*}
In other words, the deterministic limiting trajectory is quadratic.
Moreover, in this range (indeed, whenever
$t=o(n)$) we have $\beta_t\sim 1$ (see \eqref{bsim}) and, from \eqref{var},
the martingale differences appearing in \eqref{St} have (conditional)
variance
\[
 \beta_i^{-2}\Var(\Delta_i\mid\cF_{i-1}) \sim \Var(\Delta_i\mid\cF_{i-1}) =\Var(\eta_i\mid\cF_{i-1})\sim k-1.
\]
It follows using the bounds on $|X_t-x_t-S_t|$ established above
that the rescaled process whose value at rescaled time $s$
is $X_{s (k-1)^{-1/3}n^{2/3}}/((k-1)n)^{1/3}$ converges to the quadratic
function above plus a standard Brownian motion, i.e., to $W^\alpha(s)$.
The rest of the argument is exactly as in the original paper of Aldous~\cite{Aldous}, 
so we omit the details, noting only that it is the \emph{time} rescaling factor
that appears in the component sizes.
\end{proof}
For slightly more details of a related argument, see~\cite[Section 4]{Riordan_walk}.

\medskip
\noindent
{\bf Acknowledgement.}
This research was triggered by a question that Tomasz \L uczak asked us during the conference
Random Structures and Algorithms 2011 in Atlanta. We are grateful to him for raising this question.

\end{document}